\documentclass[10pt]{amsart}

\usepackage{amsfonts, amssymb,amsmath, amsthm}

\newtheorem{theorem}{Theorem}
\newtheorem{proposition}[theorem]{Proposition}

\newtheorem{question}[theorem]{Question}
\newtheorem{corollary}[theorem]{Corollary}
\newtheorem{remark}[theorem]{Remark}

\newcommand{\N}{\mathbb{N}}

\DeclareMathOperator{\dist}{\mathsf{dist}}

\begin{document}

\title[]{A note on the relation between Hartnell's firefighter problem and growth of groups}
 \author[E.~Mart\'inez-Pedroza]{Eduardo Mart\'inez-Pedroza}
 \address{Memorial University\\ St. John's, Newfoundland, Canada A1C 5S7}
 \email{emartinezped@mun.ca}

 \begin{abstract}
The firefighter game problem on locally finite connected graphs was introduced by Bert Hartnell~\cite{FHS00}. The game on a graph $G$ can be described as follows:  let $f_n$ be a sequence of positive integers; an initial fire starts at a finite set of vertices; at each (integer) time $n\geq 1$, $f_n$ vertices which are not on fire become protected, and then the fire spreads to all unprotected neighbors of vertices on fire; once a vertex is protected or is on fire, it remains so for all time intervals. The graph $G$ has the \emph{$f_n$-containment property} if every initial fire admits an strategy that protects $f_n$ vertices at time $n$ so that the set of vertices on fire is eventually constant.  If the graph $G$ has the containment property for a sequence of the form $f_n=Cn^d$, then the graph is said to have  \emph{polynomial containment}.  In~\cite{DMPT15}, it is shown that any locally finite graph with polynomial growth has  polynomial containment; and it is remarked that the converse does not hold. That article also raised the question of whether the equivalence of polynomial growth and polynomial containment holds for Cayley graphs of finitely generated groups. In this short note, we remark how the equivalence holds for elementary amenable groups and for non-amenable groups from results in the literature.
  \end{abstract}

\maketitle

Let $G$ be  a connected and locally finite graph, and let $d$ be a non-negative integer. The graph $G$ satisfies a \emph{polynomial containment of degree at most $d$} if there exists a constant $C>0$ such that any finite subset $X_0$ of vertices of $G$ admits a $\{Cn^d\}$-containment strategy, i.e.,  there is a sequence   $\{W_k\colon k\geq 1\}$ of subsets of vertices of $G$  such that
\begin{enumerate}
\item for every $n\geq 1$, the set $W_n$ has cardinality at most $Cn^d$,
\item the sets $X_{n}$ and $W_{n+1}$ are disjoint for $n\geq 0$, where  $X_n$ for $n>0$ is defined as the set of vertices at distance less than or equal to one of $X_{n-1}$  and do not belong to $W_1\cup \cdots \cup W_n$.
\item \label{item:time} there is an integer $T>0$ such that $X_n=X_T$ for every $n\geq T$.
\end{enumerate}

This definition can be interpreted as the existence of winning strategies for a single player game: a fire starts at an arbitrary finite set of vertices $X_0$, at each time interval the fire spreads to all unprotected vertices that neighbor vertices on fire, and  $X_n$ is the set of vertices on fire at time $n$; the player aims to contain the fire by protecting at time $n$ a set $W_n$ of at most  $Cn^d$ vertices that are not on fire;  once a vertex is protected or is  on fire, it remains so for all time intervals.
The graph $G$  satisfies \emph{polynomial containment of degree at most $d$} if there is $C>0$ such that for any initial fire $X_0$ the player has an $\{Cn^d\}$-strategy to contain the fire.

The notion of edge-path in a graph $G$ defines a metric on the set of vertices  by declaring $\dist_G(x, y)$ to be the length of the shortest path from $x$ to $y$. The metric $\dist$ on the set of vertices of $G$ is called \emph{the combinatorial metric on $G$.}   The notion of quasi-isometry is an equivalence relation between metric spaces which  plays a significant role in the study of discrete groups, for a definition and overview see~\cite{BrHa99}. 

\begin{theorem}\label{thm:qi-containment}\cite[Theorem 8]{DMPT15}
In the class of connected graphs with bounded degree, the property of having polynomial containment of degree at most $d$ is preserved by quasi-isometry.
\end{theorem}

For a  connected and locally finite graph $G$ with a chosen vertex $g_0$, let $\beta(n)$ denote the number of vertices at distance at most $n$ from $g_0$. Recall that the function $\beta\colon \N \to \N$ is called the growth function of $G$ with respect to $g_0$. A graph $G$ has \emph{polynomial growth of degree at most $d$} if there is a constant $C>0$ such that $\beta(n) \leq Cn^d$ for each integer $n$.  It is an exercise to show that having polynomial growth of degree $\leq d$ is independent of the base vertex $g_0$.  
Moreover, in the class of connected graphs with bounded degree, it is a quasi-isometry invariant~\cite{BrHa99}.  

The following relation between polynomial growth and polynomial containment can be used to provide examples of graphs with polynomial containment. 
\begin{theorem}\label{thm:growth-containment}\cite[Theorem 3]{DMPT15}
Let $G$ be a connected graph with polynomial growth of degree $d$. Then $G$ has polynomial containment of degree at most $\max\{ 0, d-2 \}$. 
\end{theorem}
\begin{remark}\label{rem:converse}
The converse of Theorem~\ref{thm:growth-containment} does not hold, there are connected graphs of bounded degree with subexponential growth and constant containment, see~\cite[Example 2.7]{DMPT15}. The examples in the cited paper are graphs of bounded degree with infinitely many vertices that separate the graph into a finite subgraph and an infinite subgraph. 
\end{remark}

In contrast, examples of locally finite infinite graphs without polynomial containment can be exhibited using the following proposition. By a $\delta$-regular tree, we mean an infinite tree such that every vertex has degree exactly $\delta$.
\begin{proposition} \label{prop:tree-containment}\cite[Corollary 11]{DMPT15}
If a graph $G$  contains a subgraph quasi-isometric to the infinite $\delta$-regular tree with $\delta\geq 3$, then $G$ does not satisfy a polynomial containment property. 
\end{proposition}

For any finitely generated group $G$,  any two Cayley graphs with respect to finite generating sets  are quasi-isometric. In view of Theorem~\ref{thm:qi-containment},  we say that  $G$ has \emph{polynomial containment of degree $d$} if the Cayley graph of $G$ with respect to a finite generating set has polynomial containment of degree $d$; and we say that $G$ has \emph{polynomial containment} if it has polynomial containment of degree $d$ for some $d$.  The question of whether the converse of Theorem~\ref{thm:growth-containment} holds for finitely generated groups was raised in~\cite{DMPT15}. 

\begin{question} \cite[Question 12]{DMPT15} \label{question}  In the class of finitely generated groups, 
  is having polynomial growth of degree $d$ equivalent to having polynomial containment of degree $\max\{0, d-2\}$?
\end{question}

This note discusses the following weaker version of the above question.
\begin{question}  \label{question2}  In the class of finitely generated groups,
 is having polynomial growth equivalent to have polynomial containment?
\end{question}
In this note we observe that the answer is positive for elementary amenable groups, and for non-amenable groups. Indeed, in the class of finitely generated groups, polynomial containment implies amenability; and in the class of finitely generated elementary amenable groups, polynomial containment implies polynomial growth. Below we explain this statements.

Recall that the Cheeger constant  $h(G)$ of a locally finite graph $G$ is defined as
\[ h(G) = \inf_K \frac{|\partial K|}{|K|}\]
where $K$ is any non-empty finite subset of vertices  of $G$, and $\partial K$ is the set of edges of $G$ with one endpoint in $K$ and the other endpoint not in $K$.
A locally finite graph $G$ is \emph{amenable} if $h(G)=0$, and otherwise is called \emph{non-amenable}. In the class of graphs with bounded degree, being amenable is preserved by quasi-isometry; hence for a finitely generated group either  Cayley graphs with respect to  finite generating sets are amenable, or all are non-amenable. This yields the notions of \emph{amenable} and \emph{non-amenable} in the class of finitely generated groups.

\begin{theorem}\label{thm:CheegerTree} \cite[Theorem 1.1]{BS97}
Any non-amenable locally finite graph contains a tree with positive Cheeger constant.
\end{theorem}

\begin{corollary}\label{cor:amenable}
Finitely generated groups with polynomial containment are amenable. 
\end{corollary}
\begin{proof}
Let $\Gamma$ be a non-amenable finitely generated group, and $G$ be its Cayley graph with respect to a finite generating set. 
By Theorem~\ref{thm:CheegerTree}, the graph $G$  contains a subgraph $T$ which is a tree with positive Cheeger constant. 
Observe that $T$ is necessarily infinite; since $T$ has bounded degree and $h(T)>0$, it follows that there is an upper bound on the length
of embedded paths in $T$ whose interior vertices have degree $2$ in $G$. It follows that $T$, and hence $G$, contains a subgraph quasi-isometric to the $3$-regular tree. 
Therefore Proposition~\ref{prop:tree-containment} implies that $G$  does not have polynomial containment. 
\end{proof}

The class of \emph{elementary amenable} groups is the smallest class of groups containing all finite groups and all abelian groups, and closed under taking subgroups, quotients, extensions, and direct unions. 

\begin{theorem}\label{thm:semigroup}\cite[Theorem 3.2']{Ch80}
A finitely generated elementary-amenable group  is either virtually nilpotent or contains a free semigroup in two generators.
\end{theorem}

\begin{corollary}\label{cor:elementary}
Finitely generated elementary amenable groups with polynomial containment are virtually nilpotent and hence they have polynomial growth.
\end{corollary}
\begin{proof}
Let $\Gamma$  be a finitely generated elementary amenable group, and suppose that it is not virtually nilpotent. 
By Theorem~\ref{thm:semigroup}, there are elements $a,b \in \Gamma$ generating a free semigroup $S$ in two generators.
Consider the Cayley graph $G$ of $\Gamma$ with respect to a finite generating set containing $a$ and $b$.  Then the elements of $S$ 
span a subgraph $T$ of $G$ which is an infinite rooted tree, the root is the identity element of $G$, and any other vertex has degree $3$.
Observe that $T$ is quasi-isometric to the $3$-regular tree, and hence Proposition~\ref{prop:tree-containment} implies that $G$  does not have polynomial containment.  

The result follows by invoking Wolf's result that finitely generated virtually nilpotent groups have polynomial growth~\cite{Wolf68}.
\end{proof}

From Corollaries~\ref{cor:amenable} and~\ref{cor:elementary}, it follows that  polynomial growth is equivalent to polynomial containment in the class of finitely generated groups that are elementary amenable or non-amenable.  The class of non-elementary amenable groups has been a subject of intense study since it was shown in the 1980's that the class is non-empty with the appearance of Grigorchuk's group of intermediate growth~\cite{Gr83}.   The class of non-elementary amenable groups contains groups of intermediate growth as wells as groups of exponential growth. An example of a non-elementary amenable group of exponential growth is the Basilica group;  the group was introduced by Grigorchuk and Zuk~\cite{GZ02} where they proved that it has exponential growth and it does not belong to the class of elementary amenable groups; the proof that the group is amenable was found by Bartholdi and Virag~\cite{BV05}.  It is conceivable that Proposition~\ref{prop:tree-containment} can be used to show that groups of exponential growth have no polynomial containment, in other words, that the following folk question has a positive answer. A small variation of this question was asked by A.Thom~\cite{Thom13}. 

\begin{question}
Let $G$ be a finitely generated non-elementary amenable group of exponential growth, and let $S$ be a finite generating set. Does the Cayley graph of $G$ with respect to $S$ contains a a tree with positive Cheeger constant?
\end{question}

Since the Cayley graph of a group of intermediate growth cannot contain a tree with positive Cheeger constant,  answering Question~\ref{question2} is a more subtle question.

\medskip

\noindent \textbf{Acknowledgements.} The author acknowledges funding by the Natural Sciences and Engineering Research Council of Canada (NSERC), and thanks  Gidi Amir, Florian Lehner,  Pierre Will, and the  referee  for comments and proofreading.

\end{document}